\documentclass[reqno]{amsart}
\usepackage[utf8]{inputenc}
\usepackage[T1]{fontenc}
\usepackage{lmodern}
\usepackage{hyperref}
\usepackage{amsmath}
\usepackage{amsfonts}
\usepackage{amssymb}
\usepackage{cancel}
\usepackage{color}
\usepackage{lipsum}
\numberwithin{equation}{section}
\newtheorem{theorem}{Theorem}[section]

\newtheorem{claim}{Claim}[section]
\newtheorem{lemma}{Lemma}[section]
\newtheorem{example}{Example}[section]
\newcommand{\<}{\langle}
\renewcommand{\>}{\rangle}
\begin{document} 
\title{ Linear elliptic system with nonlinear boundary conditions without Landesman-Lazer conditions }

\author[Alzaki Fadlallah\hfil ]
{Alzaki Fadlallah}

\address{Alzaki.M.M. Fadlallah \hfill\break
Department of Mathematics, University of Alabama at
Birmingham \hfill\break
Birmingham, Alabama 35294-1170, USA }
\email{ zakima99@@math.uab.edu }


\begin{abstract}
The boundary value problem is examined for the system of elliptic equations of from 
$-\Delta u + A(x)u = 0 \quad\text{in } \Omega,$ where  $A(x)$ is positive semidefinite matrix on $\mathbb{R}^{{k}\times{k}},$ and 
$\frac{\partial u}{\partial \nu}+g(u)=h(x) \quad\text{on } \partial\Omega$ It is assumed 
that $g\in C(\mathbb{R}^{k},\mathbb{R}^{k})$ is a bounded function which may vanish at infinity. The proofs are based on Leray-Schauder 
degree methods.

\end{abstract}
\maketitle

\section{Introduction}
Let $\mathbb{R}^{k}$ be real $k-$dimensional space, if $w\in\mathbb{R}^{k}$, then $|w|_{E}$ denotes the Euclidean norm of $w.$
Let $\Omega\subset\mathbb{R}^N$ , $N\geq 2$ is a bounded domain  with boundary
$\partial \Omega$ of class $C^{\infty}$. Let $g\in C^{1}(\mathbb{R}^{k},\mathbb{R}^{k}),$ $h\in C(\partial\Omega,\mathbb{R}^{k}),$
and the matrix 
$$A(x)=\left[
\begin{array}{cccc}
 a_{1}(x)&0&\cdots&0\\ 
0&a_{2}(x)&\cdots&0)\\ 
\vdots & \vdots  &\ddots &\vdots  \\ 
0&0&\cdots&a_{k}(x)
\end{array}\right]
$$\\
Verifies the following conditions:
\begin{enumerate}
\item[(A1)] The functions $a_{i}:\Omega\to\mathbb{R},$ $a_{i}(x)\ge 0,~~ \forall i=1,\cdots,k ~~x\in\Omega$
with strict inequality on a set of positive measure.
 \item[(A2)] $A(x)$ is positive semidefinite matrix on $\mathbb{R}^{{k}\times{k}},$ almost everywhere $x\in\Omega,$ and 
 $A(x)$ is positive definite on a set of positive measure with $a_{ij}\in L^{p}(\Omega)~~ \forall ~i=1,\cdots,k$ 
for $p>\frac{N}{2}$ when $N\geq 3$, and $p>1$ when $N=2$
 \end{enumerate}
 We will study the solvability of 
\begin{equation}\label{e1}
\begin{gathered}
 -\Delta u + A(x)u = 0 \quad\text{in } \Omega,\\
\frac{\partial u}{\partial \nu}+g(u)=h(x) \quad\text{on }
\partial\Omega,
\end{gathered}
\end{equation}
The interest in this problem is the resonance case at the boundary with a bounded nonlinearity, we will assume that $g$ a bounded function, 
and there is a constant  $R>0$ such that 
\begin{equation}\label{e2}
\begin{gathered}
 |g(w(x))|_{E}\leq R \quad\ \forall ~ w\in \mathbb{R}^{k} ~~\& ~~  x\in \partial\Omega
\end{gathered}
\end{equation}
Our assumptions allow that $g$ is not only bounded, but also may be vanish at infinity i.e.; 
\begin{equation}\label{e3}
\begin{gathered}
 \displaystyle\lim_{{|w|_{E}}\to\infty}g(w)=0\in \mathbb{R}^{k} 
 \end{gathered}
\end{equation}
Condition (\ref{e3})is not required by Our assumptions, but allowing for it is the main result of this paper.\\
In case of the scalar equation $k=1$ and $g$ doesn't satisfy condition (\ref{e3})
but satisfying  the Landesman-Lazer condition 
$$g_{-}<\bar{h}<g^{+}$$ where 
$\displaystyle\lim_{w\to-\infty}g(w)=g_{-},$ $\bar{h}=\frac{1}{|\partial\Omega|}\int_{\partial\Omega}h\,dx$, 
$\displaystyle\lim_{w\to\infty}g(w)=g^{+}$\\
 And $A(x)=0\in\mathbb{R}^{k\times k}$
 Then it is well know that there is a solution (\ref{e1}). 
The first results when the nonlinearity in the equation in scalar case was done  by Landesman and Lazaer \cite{LLC} in $1970$, 
Their work led  to great interest and activity on boundary value problems at resonance which continuous to this day. 
A particularly interesting extension of Landesman and Lazer's work to systems was done by Nirenberg \cite{LN}, \cite{LN1} 
in case of system and the nonlinearity in the equation was  done by Ortega and Ward \cite{OW}, in the scalar case without Landesman-Lazer condition 
was done by Iannacci and Nkashama \cite{IK1}, Ortega and S\'{a}nchez \cite{OS}, more completely the case for periodic solutions of the system 
of ordinary differential equations with bounded nonlinear $g$ satisfying Nirenberg's condition. They studied periodic so solutions 
$$u''+cu'+g(u)=p(t)$$
for $u\in\mathbb{R}^{k}.$\\
 In case $c=0$ was done by Mawhin \cite{M2}. In case the nonlinear terms vanish at infinity, as in (\ref{e3}), the Landesman-Lazer
 conditions fail. We would like to know what we can do in this case, and what conditions on a bounded nonlinearity that 
 vanishes at infinity might replace that ones of the Landesman-Lazer type. Several authors have considered the case when the nonlinearity
$g\colon\partial\Omega\times\mathbb{R}\rightarrow\mathbb{R}$ is a scalar function
satisfies Carath\'{e}odory conditions
i,e.; 
\begin{description}
 \item[i] $g(.,u)$ is measurable on $\partial\Omega$, for each $u\in\mathbb{R}$,
\item[ii] $g(x,.)$ is continuous on $\mathbb{R},$ for $a.e.x\in \partial\Omega$,
\item[iii]for any constant $r>0$, there exists a function \\$\gamma_{r}\in L^{2}(\partial\Omega),$ such that 
\begin{equation}\label{00}
\begin{gathered}
|g(x,u)|\leq\gamma_{r}(x),
\end{gathered}
\end{equation}

for $a.e.x\in \Omega$, and all $u\in\mathbb{R}$ with $|u|\leq r,$
\end{description}
Was done by Fadlallah \cite{AF}
 and the others have considered the case when the nonlinearity does not decay to zero very rapidly. For example in case the nonlinearity in the equation
if $g=g(t)$ is a scalar function, the condition 
\begin{equation}\label{01}
\begin{gathered}
\lim_{|t|\to\infty}tg(t)>0
\end{gathered}
\end{equation}
 and related ones were assumed in \cite{Amb}, \cite{Ambr}, \cite{Au}, \cite{FK}, \cite{PH}, \cite{MK1}, \cite{MK2}, \cite{MN}, \cite{Ru}. These 
 papers all considered scalar problem, but also considered the Dirichlet (Neumann) problem  at resonance (non-resonance) at higher 
 eigenvalues (Steklov-eigenproblems). The work in some  of these papers makes use of Leray-Schauder degree arguments, and the others using 
 critical point theory both the growth restrictions like \ref{01} and Lipschitz conditions have been removed (see \cite{MK2}, \cite{Ru}).
 In this paper we study systems of elliptic boundary value problems with nonlinear boundary conditions Neumann type and the nonlinearities at boundary 
 vanishing at the infinity.  We do not require the problem to be in variational from. 
 \subsection{Assumptions}
 \begin{description}
  \item[G1] $g\in C^{1}(\mathbb{R}^{k},\mathbb{R}^{k})$ and $g$ is bounded with $g(w)\neq 0$  for $|w|_{E}$ large. Let $S^{k-1}$
  be the unit sphere in $\mathbb{R}^{k}$ 
  \item[G2] We will assume that $S^{k-1}\cap\partial\Omega\neq\emptyset$ and Let $\mathbb{S}=S^{k-1}\cap\partial\Omega$
  \item[G3] For each $z\in\mathbb{S}$ the $\displaystyle\lim_{r\to\infty}\frac{g(rz)}{|g(rz)|_{E}}=\varphi(z)$ exists, and the limits is uniform for 
  $z\in \mathbb{S}.$
  \item It follows that $\varphi\in C(\mathbb{S},\mathbb{S})$ and the topological degree of $\varphi$ is defined.
  \item[G4] $deg(\varphi)\neq 0$
 \end{description}
 \subsection{Notations}
 \begin{itemize}

\item  Let $\<.,.\>_{L^{2}}$ denote the inner product in $L^{2}:=L^{2}(\Omega,\mathbb{R}^{k})$  where $L^{2}$ is Lebesgue space 
 \item Let $\<.,.\>_{E}$ denote the standard inner product in $\mathbb{R}^{k}$
  for $u,v\in H^{1}=H^{1}(\Omega,\mathbb{R}^{k})$ where $H^{1}$ the Sobolev space 
  \end{itemize}
  We note that if follows from the assumptions G1-G4 that on large balls 
$$B(R):=\{y:|y|_{E}\leq R\},$$ the $deg(g,B(R),0)\neq 0$ see \cite{L},\cite{Ma}. \\
We modify the Lemma 1 and Theorem 1 \cite{OW} to fit our problem. 
\begin{lemma}\label{le1}
 assume that G1 and G3 hold and $C>0$ is a given constant. Then there is $R>0$ such that 
 $$\int_{\partial\Omega}g(u(x))\,dx\neq 0,$$
 for each function $u\in C(\partial\Omega,\mathbb{R}^{k})$ (we can write $u=\bar{u}+\tilde{u}$ where $\bar{u}=\int_{\partial\Omega}u(x)dx=0,$ 
 and $\bar{u}\bot\tilde{u}$) with $|\bar{u}|_{E}\geq R$ and $||u-\bar{u}||_{L^{\infty}(\partial\Omega)}\leq C$
\end{lemma}
\begin{proof}
 By the way of contradiction. Assume that for some $C>0$ there is exist a sequence of functions 
 $\{u_{n}\}_{n=1}^{\infty}\in C(\bar{\Omega},\mathbb{R}^{k}),$ with 
 $$|\bar{u}_{n}|_{E}\to\infty,~~||u_{n}-\bar{u}_{n}||_{L^{\infty}(\partial\Omega)}\leq C$$ and 
 \begin{equation}\label{02}
\begin{gathered}
\int_{\partial\Omega}g(u_{n}(x))\,dx=0
\end{gathered}
\end{equation}

We constructed a subsequence of $u_n$ one can assume that $\bar{z}_{n}=\frac{\bar{u}_n}{|\bar{u}_{n}|_{E}}$ converges to some 
point $z\in\mathbb{S}.$ The uniform bound on $u_{n}-\bar{u}_{n}$ implies that also $\frac{u_{n}}{|u_{n}|_{E}}$ converges to $z$ and this convergence 
is uniform with respect to $x\in\bar{\Omega}.$ It follows  from the assumption G3 that 
$$\displaystyle\lim_{n\to\infty}\frac{g(u_{n}(x)}{|g(u_{n}(x))|_{E}}=\varphi(z)$$
uniformly in $\bar{\Omega}.$ Since $\varphi(z)$ is in the unit sphere one can find an integer $n_{0}$ such that if 
$n\geq n_{0}$ and $x\in\bar{\Omega}$, then 
$$\<\frac{g(u_{n}(x)}{|g(u_{n}(x))|_{E}},\varphi(z)\>_{E}\geq \frac{1}{4}$$
Define 
$$\gamma_{n}(x)=|g(u_{n}(x))|_{E}.$$ By G1 clearly $\gamma_{n}>0$ everywhere. 
For $n\geq n_{0}$
$$\<\int_{\partial\Omega}g(u_{n}(x))\,dx,\varphi(z)\>_{E}=\int_{\partial\Omega}\<g(u_{n}(x)),\varphi(z)\>_{E}\,dx$$
$$=\int_{\partial\Omega}\gamma_{n}(x)\<\frac{g(u_{n}(x))}{\gamma_{n}(x)},\varphi(z)\>_{E}\,dx\geq\frac{1}{4}\int_{\partial\Omega}\gamma_{n}(x)\,dx>0$$
Therefore $\displaystyle\int_{\partial\Omega}g(u_{n}(x))\,dx>0.$ 
Now we have contradiction with (\ref{02})\\
The proof completely of the lemma.
\end{proof}
\section{Main Result}
Let \begin{equation}\label{N1}
\begin{gathered}
Qu=Nu
\end{gathered}
\end{equation}
be a semilinear elliptic boundary value problem. Suppose $N$ is continuous and
bounded (i.e.,$|Nu|_{E}\leq C$ for all $u$).
If $Q$ has a compact inverse $Q^{-1}$ then by Leray-Schauder theory (\ref{N1}) has a solution. On the other hand if
$Q$ is not invertible the existence of a solution depends on the behavior of $N$ and its interaction with the
null space of $Q$ see \cite{Ma}. 
\begin{theorem}\label{th}
 Assume  That $g\in C^{1}(\mathbb{R}^{k},\mathbb{R}^{k})$ satisfies G1, G3, and G4. If $h\in C(\partial\Omega,\mathbb{R}^{k}),$
 satisfies $\bar{h}=0$ then (\ref{e1}) has at least one solution.
\end{theorem}
\begin{proof}
Define 
 $$Dom(L):=\{u\in H^{1}(\Omega): -\Delta u + A(x)u = 0\} $$
 Define an operator $L$ on  $L^{2}=L^{2}(\Omega,\mathbb{R}{k})$ for $u\in Dom(L)$ and each $v\in H^{1}(\Omega)$ by 
 $$Lu=\frac{\partial u}{\partial\nu},~~~\forall~~u\in Dom(L)$$
 we use the embedding theorem see \cite{Bz}since you know that $H^{1}(\Omega)\hookrightarrow L^{2}(\Omega)$
 and the trace theorem $H^{1}\to L^{2}(\partial\Omega)$, thus  $L:Dom(L)\subset L^(\partial\Omega)\to L^{2}(\partial\Omega)$ then the 
 equation 
 $$<Lu,v>=<h,v>_{L^{2}(\partial\Omega)}~~~\forall~~v\in H^{1}(\partial\Omega)$$
 if and only if 
 $$Lu=h.$$
 The latter equation is solvable if and only if 
 $$Ph:=\frac{1}{|\partial\Omega|}\int_{\partial\Omega}h=0$$
 Now if $h\in L^{\infty}(\partial\Omega,\mathbb{R}^{k})$ and $Ph=0$ then each solution $u\in H^{1}(\Omega)$ is 
 H\"{o}lder continuous, so $u\in C^{\gamma}(\bar{\Omega},\mathbb{R}^{k})$ for some $\gamma\in (0,1).$
 Since we know that there is constant $r_{1}$ such that 
 $$||u||_{\gamma}\leq r_{1}\left(||u||_{L^{2}(\partial\Omega)}+||h||_{L^{\infty}(\partial\Omega)}\right)$$
 When $Ph=0$ there is a unique solution $Kh=\tilde{u}\in H^{1}(\Omega)$ with $P\tilde{u}=0$
 to 
 $$Lu=h,$$
and if $h\in C(\partial\Omega)=C(\partial\Omega,\mathbb{R}^{k})$ then
$$||Kh||_{\gamma}\leq r_{1}\left(||Kh||_{L^{2}(\partial\Omega)}+||h||_{L^{\infty}(\partial\Omega)}\right)\leq r_{2}||h||_{C(\partial\Omega)}$$
and $K$ maps $C(\partial\Omega)$ into itself  take compact set to compact set i.e.; compactly.\\
Let $Q$ be the restriction of $L$ to $L^{-1}(C(\partial\Omega))=KC(\partial\Omega)+\mathbb{R}^{k}.$
We  define $N:C(\partial\Omega)\to C(\partial\Omega)$ define by 
$$N(w)(x):=h(x)-g(w(x))~~\forall w\in C(\partial\Omega)$$
is continuous. Now (\ref{e1}) can be written as 
$$Qu=Nu$$
and $\ker Q=\Im P,$ $\Im Q=\ker P.$ The linear map $Q$ is a Fredholm map (see \cite{MN}) and $N$ is $Q-$compact (see \cite{Ma}).
Now we define the Homotopy equation as follows 
Let $\lambda \in [0,1]$ such that 
\begin{equation}\label{H}
\begin{gathered}
Qu=\lambda Nu.
\end{gathered}
\end{equation}
The a priori estimates (i.e.; the possible solutions of (\ref{H}) are uniformly bounded in $C(\partial\Omega)$) 
Now we show that the possible solutions of (\ref{H}) are uniformly bounded in $C(\partial\Omega)$
independent of $\lambda\in [0,1]$
Since we know that $u=\bar{u}+\tilde{u}$ where $\bar{u}=Pu.$ Then
$$||\tilde{u}||_{\gamma}=||\lambda KNu||_{\gamma}\leq r_{2}||Nu||_{C(\partial\Omega)}\leq R_{1}$$
Where $R_{1}$ is a constant ($g$ is abounded function). 
It remains to show that $\bar{u}\in\mathbb{R}^{k}$ is bounded, independent of $\lambda\in[0,1]$. 
By the way of contradiction assume is not the case (i.e.; $\bar{u}$ unbounded). Then there are sequence $\{\lambda_{n}\}\subset[0,1],$
and $\{u_{n}\}\subset Dom (Q)$ with 
$||\tilde{u}_{n}||_{\gamma}\leq R_{1},$
$$Qu_{n}=\lambda_{n}Nu_n~~{\rm and } ~~|\bar{u}_{n}|_{E}\to\infty$$
We get that 
$$PNu_n=PN(\tilde{u}_n+\bar{u}_n)=-\int_{\partial\Omega}g(\tilde{u}_{n}(x)+\bar{u}_{n}(x))\,dx=0$$
Now $u_{n}=\tilde{u}_n+\bar{u}_n$ so $||{u}_n-\bar{u}_n||_{L^{\infty}(\partial\Omega)}=||\tilde{u}_{n}||_{L^{\infty}(\partial\Omega)}\leq R_{1}$
and $||\bar{u}_{n}||_{L^{\infty}(\partial\Omega)}\to\infty$. It follows from Lemma\ref{le1} that for all sufficiently large $n$ 
$$\int_{\partial\Omega}g(u_{n}(x))\,dx\neq 0$$
We have reached a contradiction, and hence all possible solutions of (\ref{H}) are uniformly bounded in $C(\partial\Omega)$ independent of $\lambda\in[0,1]$
\\
Let $\bar{B}(0,r)=\{x:|x|_{E}\leq r\}$ denote the ball in $C(\partial\Omega,\mathbb{R}^{k})$ 
Now you  can apply Leray-Schauder degree theorem see (\cite{L},\cite{Ma}), the only thing left to show is that
$$deg(PN,\bar{B}(0,r)\cap\ker Q,0)\neq0,$$ for large $r>0.$
So 
$deg(PN,\bar{B}(0,r)\cap\ker Q,0)=deg(g,\bar{B}_{r},0),$ where $\bar{B}_{r}$ is the ball in $\mathbb{R}^{k}$ of radius $r$. Since for $|x|_{E}$
large, and $deg(\varphi)\neq0$ we have that 
$deg(g,\bar{B}_{r},0)\neq 0$ for large $r$.
Therefore $deg(PN,\bar{B}(0,r)\cap\ker Q,0)\neq 0$
By Leray-Schauder degree theorem equation (\ref{H}) has a solution when $\lambda=1.$ Therefore equation (\ref{e1}) has at least one solution.
This proves the theorem.
\end{proof}
We will give one Example 
\begin{example}
 Let $\Omega\subset\mathbb{R}^N$ , $N\geq 2$ is a bounded domain  with boundary
$\partial \Omega$ of class $C^{\infty}$. Let 
\begin{equation}\label{ex}
\begin{gathered}
-\Delta u + A(x)u = 0 \quad\text{in } \Omega,\\
\frac{\partial u}{\partial \nu}+\frac{u}{1+|u|_{E}^{3}}=h(x) \quad\text{on } \partial\Omega
\end{gathered}
\end{equation}
 where  $A(x)$ is positive semidefinite matrix on $\mathbb{R}^{{2}\times{2}},$ and 
Where $u=(u_1,u_2)\in \mathbb{R}^{2}$ and $h$ real valued function and continuous on $\partial\Omega,$ and 
$\int_{\partial\Omega}h(x)\,dx=0$
and $g(u)=\frac{u}{1+|u|_{E}^{2}}$  
$$\displaystyle\lim_{u\to\infty}g(u)=\displaystyle\lim_{|u|_{E}\to\infty}\frac{u}{1+|u|_{E}^{2}}=0$$
$g(u)$ vanishes at infinity,  clearly $g\in C^{1}(\mathbb{R}^{2},\mathbb{R}^{2})$ and bounded with $g(u)\neq0,$ for $|u|_{E}$ large.
Therefore $g$ satisfies G1.\\
$$\frac{g(ru_1,ru_2)}{|g(ru_1,ru_2)|}=\frac{g(ru)}{|g(ru)|}=\frac{\frac{ru}{1+|ru|_{E}^{2}}}{\left|\frac{ru}{1+|ru|_{E}^{2}}\right|}=
\frac{u}{|u|_{E}^{2}}=u$$
For all $u$ in $\mathbb{S}$ and $r>0$. Therefore G3 holds.\\
And $\varphi(u)=u$ so that $deg(\varphi)\neq 0.$Therefore G4 holds. By theorem\ref{th} (\ref{ex}) has at least one solution. 
\end{example}

\end{document}